\documentclass[11pt,bezier]{article}
\usepackage{amsmath,amssymb,amsfonts,euscript,mathtools}
\usepackage{amsthm}
\usepackage{enumerate}
\usepackage{cite}
\usepackage[bookmarks]{hyperref}
\usepackage{cleveref}

\usepackage{ifpdf}
\ifpdf    
\usepackage{hyperref}
\else    
\usepackage[hypertex]{hyperref}
\fi
\usepackage{empheq}

\DeclarePairedDelimiterX{\inp}[2]{\langle}{\rangle}{#1, #2}

\newtheorem{prethm}{{\bf Theorem}}

\newenvironment{thm}{\begin{prethm}{\hspace{ - 0.5
				em}}}{\end{prethm}}

\newtheorem{preconj}[prethm]{Conjecture}
\newenvironment{conj}{\begin{preconj}{\hspace{ - 0.5
				em}}}{\end{preconj}}

\newtheorem{preproblem}{{\bf problem}}

\newcommand{\Gb}{{\overline G}}
\newcommand{\Hb}{{\overline H}}

\title{\large \bf The Algebraic Connectivity of a Graph and its Complement}

\author{{\normalsize
		{\sc B. Afshari${}^{\mathsf{a}}$},\,
		{\sc S. Akbari${}^{\mathsf{b}}$}
		\thanks{The research of the second author was partly funded by Iran National Science Foundation (INSF) under the contract No. 96004167.} ,\,
		{\sc M.J. Moghaddamzadeh${}^{\mathsf{b}}$},\,
		{\sc B. Mohar${}^{\mathsf{c}}$}
		\thanks{B.M.~was supported in part by the NSERC Discovery Grant R611450 (Canada), by the Canada Research Chairs program, and by the Research Project J1-8130 of ARRS (Slovenia).}
		\thanks{On leave from IMFM \& FMF, Department of Mathematics, University of Ljubljana.} \,
	}
	\vspace{3mm}
	\\{\footnotesize{${}^{\mathsf{a}}$\it School of Computer Science, Institute for Research in Fundamental Sciences,}} {\footnotesize{}}\\
	{\footnotesize{${}^{\mathsf{b}}$\it Department of Mathematical Sciences, Sharif University of Technology, Tehran,
			Iran,}}{\footnotesize{}}\\
	{\footnotesize{${}^{\mathsf{c}}$\it Department of Mathematics, Simon Fraser University, Burnaby, BC \ V5A 1S6}}{\footnotesize{}}
}

\date{}
\begin{document}

\maketitle

\begin{abstract}
	For a graph $G$, let $\lambda_2(G)$ denote its second smallest Laplacian eigenvalue.
	It was conjectured that $\lambda_2(G) + \lambda_2(\Gb) \ge 1$, where $ \Gb $ is the complement of $G$.
	In this paper, it is shown that $\max\left\{\lambda_2(G), \lambda_2(\Gb)\right\} \ge \frac{2}{5}$.
\end{abstract}


\section{Introduction}

All graphs considered in this paper are simple (no loops and no multiple edges).
If $G$ is a graph and $v\in V(G)$, we denote by $ N_{G}(v) $ the set of vertices adjacent to $ v $.
We denote the complementary graph of $ G $ by $ \Gb $.

The adjacency matrix $A(G)$ of $G$ is the matrix whose $(u,v)$-entry is equal to $1$ if $uv\in E(G)$ and $0$ otherwise.
If $D(G)$ denotes the diagonal matrix of vertex degrees, then the Laplacian of the graph $G$ is defined as $L(G) = D(G) - A(G)$. We denote the Laplacian eigenvalues of $G$ by
$$0=\lambda_1(G) \le \lambda_2(G) \le \cdots \le \lambda_n(G).$$
The second smallest eigenvalue $\lambda_2(G)$ is also called the \textit{algebraic connectivity} of $ G $ and is an important indicator related to various properties of the graph.
It is well-known that the eigenvalues of the complementary graph $L(\Gb)$ are
$$0 = \lambda_1(\Gb) \le n - \lambda_n(G) \le n - \lambda_{n - 1}(G) \le \ldots \le n - \lambda_{2}(G).$$


The \textit{Laplacian spread} of a graph $ G $ is defined to be $\lambda_n(G) - \lambda_2(G)$.
Clearly, $\lambda_n(G) - \lambda_2(G) \le n$. It was conjectured \cite{yl,zsh} that this quantity is at most $n-1$.

\begin{conj}
	For any graph $ G $ of order $ n\ge2 $, the following holds:
	$$\lambda_n(G) - \lambda_2(G) \le n-1,$$
	or equivalently $\lambda_2(G) + \lambda_2(\Gb) \ge 1$, with equality if and only if $ G $ or $ \Gb $ is isomorphic to the join of an isolated vertex and a disconnected graph of order $ n-1 $.
\end{conj}

This conjecture holds for trees \cite{fxwl}, unicyclic graphs \cite{btf}, bicyclic graphs \cite{flt,lsl,lw}, tricyclic graphs \cite{cw}, cactus graphs \cite{l}, quasi-tree graphs \cite{xm}, graphs with diameter not equal to $3$ \cite{zsh}, bipartite and $K_3$-free graphs \cite{cd}.

In this paper, we provide a positive constant lower bound for $\lambda_2(G) + \lambda_2(\Gb)$ by proving the following.

\begin{thm}\label{thm:main}
	Let $G$ be a graph of order $n\ge2$. Then
	$$ \max\left\{\lambda_2(G), \lambda_2(\Gb)\right\} \ge \frac{2}{5}.$$
\end{thm}

\section{Routings}

If $P$ is a path in a graph $G$, we denote its length by $\Vert P\Vert$.
Let $\mathcal{P}$ be a set of paths in $G$. For any edge $e\in G$, let $w_{\mathcal{P}}(e)$ be the sum of the lengths of all paths in $\mathcal{P}$ which contain $e$.
We say $\mathcal{P}$ has \emph{weighted congestion} $w$, where $w = \max_{e\in E(G)} {w_{\mathcal{P}}(e)}$.
We denote the weighted congestion of $\mathcal{P}$ by $w(\mathcal{P})$.

A set $\mathcal{P}$ of paths in $G$ is called a \textit{routing} if for any distinct vertices $x,y \in V(G)$, there is \textit{exactly one} path $P_{xy}\in \mathcal{P}$ with endpoints $x$ and $y$. In particular, this means that $P_{xy}=P_{yx}$.

\begin{thm}\label{SystemOfPaths&Lambda}
	Let $G$ be a graph of order $n$. If $G$ has a routing $\mathcal{P}$ of weighted congestion at most $w$, then $\lambda_2(G) \ge \frac{n}{w}$.
\end{thm}

\begin{proof}
	Let $f: V(G) \rightarrow \mathbb{R}$ be an eigenvector of $L(G)$ corresponding to the eigenvalue $\lambda_2(G)$. Then $f$ is orthogonal to the all-one vector $\mathbf{1}$, since $\mathbf{1}$ is an eigenvector corresponding to $\lambda_1(G)$. This means that $\sum_{x\in V(G)}f(x)=0$.

	Note that we have
	\begin{equation}\label{eq:1}
		\Vert f \Vert^2
		= \frac{1}{n} \sum_{ \{x,y\} \subseteq V(G) } (f(x) - f(y))^2 ,
	\end{equation}
	because
	\begin{eqnarray*}
		2 \sum_{ \{x,y\} \subseteq V(G) } (f(x) - f(y))^2
		&=& \sum_{x,y \in V(G)} (f(x) - f(y))^2 \\
		&=& \sum_{x,y \in V(G)} (f(x)^2 + f(y)^2 - 2 f(x) f(y)) \\
		&=& 2n \sum_{x \in V(G)} f(x)^2 - 2 \biggl( \sum_{x \in V(G)} f(x) \biggr)^2 \\[1mm]
		&=& 2n \Vert f \Vert^2 .
	\end{eqnarray*}
	For any distinct vertices $x,y \in V(G)$, let $P_{xy}\in \mathcal{P}$ be the path in $\mathcal{P}$ with end points $x$ and $y$. We may assume that $\Vert f \Vert = 1$. Then,
	\begin{align*}
		\lambda_2(G)
		 & = \sum_{xy \in E(G)} (f(x) - f(y))^2                                                                                                                                                                 \\
		 & = \frac{n \sum_{xy \in E(G)}{(f(x) - f(y))^2}}{ \sum_{\{x,y\} \subseteq V(G) } (f(x) - f(y))^2}                                                            & \text{ (by (\ref{eq:1}))}               \\
		 & \ge \frac{n \sum_{xy \in E(G)}{(f(x) - f(y))^2}}{ \sum_{\{x,y\} \subseteq V(G) }{\left( \Vert P_{xy}\Vert \sum_{uv \in P_{xy}} (f(u) - f(v))^2 \right) } } & \text{ (by Cauchy–-Schwarz inequality)} \\
		 & = \frac{n \sum_{xy \in E(G)}{(f(x) - f(y))^2}}{\sum_{uv \in E(G)} w_{\mathcal{P}}(uv) (f(u) - f(v))^2}                                                                                               \\
		 & \ge \frac{n}{w}.
	\end{align*}
	The proof is complete.
\end{proof}

\section{Main Result}

Now, we prove the main result of this paper.

\begin{thm}\label{mainThm}
	Let $G$ be a graph of order $n \ge 2$. At least one of $G$ or $\Gb$ has a routing of weighted congestion at most $5n/2$.
\end{thm}

\begin{proof}
	The proof is by induction on $n$. If $G$ has diameter at most 2, we select for each pair $\{x,y\}$ a path of length 1 or 2 joining $x$ and $y$.
	For any edge $e=xy\in E(G)$, the paths through $e$ are the edge $xy$, some paths from $x$ to $V(G)\setminus (N_G(x)\cup\{x\})$ and some paths from $y$ to $N_G(x)\setminus\{y\}$. Thus the weighted congestion of $e$ is at most $1 + 2(n-1-|N_G(x)|) + 2(|N_G(x)|-1) = 2n-3$.
	Thus, this routing has weighted congestion at most $2n-3$. Hence, we may assume from now on that neither $G$ nor $\Gb$ has diameter less than 3. This implies that both, $G$ and $\Gb$ have diameter exactly 3 and that $n\ge 4$.

	Let $u,v$ be vertices whose distance in $\Gb$ is 3, and let $u',v'$ have distance 3 in $G$. Then
	$uv \in E(G)$, $u'v' \in E(\Gb)$ and $N_G(u) \cup N_G(v) = N_\Gb(u') \cup N_\Gb(v') = V(G)$.
	This implies in particular that in $G$, every vertex is at distance at most 2 from $u$.
	Thus $u\notin \{u',v'\}$. By symmetry, we conclude that $\{u,v\}\cap\{u',v'\} = \emptyset$.
	Let
	\begin{align*}
		X \coloneqq N_G(u) \setminus (N_G(v) \cup \{v\}), \quad
		Y \coloneqq N_G(v) \setminus (N_G(u) \cup \{u\}), \quad
		Z \coloneqq N_G(u) \cap N_G(v), \\
		X' \coloneqq N_\Gb(u') \setminus (N_\Gb(v') \cup \{v'\}), \quad
		Y' \coloneqq N_\Gb(v') \setminus (N_\Gb(u') \cup \{u'\}), \quad
		Z' \coloneqq N_\Gb(u') \cap N_\Gb(v').
	\end{align*}
	Without loss of generality assume that $|X| \le |Y|$ and $|X'| \le |Y'|$.

	Since $u'$ and $v'$ are at distance 3 in $G$, we have
	\begin{equation}\label{eq1}
		|\{u',v'\} \cap X| = |\{u',v'\} \cap Y| = 1.
	\end{equation}
	We may thus assume that $u'\in X$ and $v'\in Y$. Similarly, we may assume that $u\in X'$ and $v\in Y'$.

	Let $H = G -\{v,v'\}$. By induction, at least one of $H$ or $\Hb$ has a routing of weighted congestion at most $5(n-2)/2$.

	Let us first assume that $H$ has such a routing $\mathcal{P}_H$.
	Let $\mathcal{A}$ be the set of following paths taken for every $z\in V(H)$:
	$$ P_{vz} = \begin{cases}
			v z   & \text{if $z \in Y \cup Z \cup \{u\}$}, \\
			v u z & \text{if $z \in X$}.
		\end{cases} $$

	We have
	$$ w_{\mathcal{A}}(e) = \begin{cases}
			1        & \text{if $e = vz$ and $z \in Y \cup Z$}, \\
			2|X| + 1 & \text{if $e = v u$},                     \\
			2        & \text{if $e = u z$ and $z \in X$},       \\
			0        & \text{otherwise}.
		\end{cases} $$

	Let $\mathcal{B}$ be the following set of paths taken for each $z\in V(H)\cup \{v\}$:
	$$ P_{v'z} = \begin{cases}
			v' z     & \text{if $z=v$},                       \\
			v' v z   & \text{if $z \in Y \cup Z \cup \{u\}$}, \\
			v' v u z & \text{if $z \in X$}.
		\end{cases} $$
	We have
	$$ w_{\mathcal{B}}(e) = \begin{cases}
			2n-3+|X| & \text{if $e = v' v$},                                       \\
			2        & \text{if $e = v z$ and $z \in (Y \cup Z)\setminus\{v'\} $}, \\
			3|X| + 2 & \text{if $e = v u$},                                        \\
			3        & \text{if $e = u z$ and $z \in X$},                          \\
			0        & \text{otherwise}.
		\end{cases} $$
	The set of paths $ \mathcal{P} = \mathcal{P}_H \cup \mathcal{A} \cup \mathcal{B} $ is a routing in $G$ with weighted congestion at most $5n/2$ since $|X| \le \frac{n-2}{2}$ and $w(\mathcal{P}_H) \le 5(n-2)/2$. Note that all we needed for this conclusion was that we had a routing in $H$ and that $v'\in Y$ (since we used the fact that $|X|\le |Y|$ to conclude that $|X|\le (n-2)/2$).

	Suppose now that the requested routing exists in $\Hb$. Since $v\in Y'$, the same proof as above shows that we can obtain a routing in $\Gb$ with weighted congestion at most $5n/2$. This completes the proof.
\end{proof}

\begin{proof}[Proof of Theorem \ref{thm:main}]
	By Theorems~\ref{SystemOfPaths&Lambda} and \ref{mainThm}, every graph $G$ of order $n\ge2$ satisfies:
	\begin{equation*}
		\max\left\{\lambda_2(G), \lambda_2(\Gb)\right\}
		\ge \frac{n}{5n/2} = \frac{2}{5}
	\end{equation*}
	which is what we were to prove.
\end{proof}

At the end, we pose the following question:

\medskip

\noindent
\textbf{Question.} {\it What is the supremum of all real numbers $c$ such that for any graph $G$ of order at least 2,}
$$ \max\left\{\lambda_2(G), \lambda_2(\Gb)\right\} \ge c.$$

\medskip

The path $P_4$ (which is self-complementary) has $\lambda_2(P_4) = 2-\sqrt{2} < 0.5858$. This shows that the supremum will be smaller than $0.5858$.

\bibliographystyle{amsplain}
%

\end{document}